\documentclass[11pt]{amsart}
\usepackage{geometry}                
\usepackage{graphicx}
\usepackage{amssymb}
\usepackage{epstopdf}
\usepackage{parskip}
\usepackage{mathtools}
\usepackage{txfonts}
\usepackage{url}
\newtheorem{theorem}{Theorem}[section]
\newtheorem{lemma}[theorem]{Lemma}
\newtheorem{proposition}[theorem]{Proposition}

\newtheorem{definition}[theorem]{Definition}

\title{Quartic Equations with Trivial Solutions over Gaussian Integers}
\author{Felix Sidokhine}

\begin{document}
\maketitle

\begin{abstract}
In our work we study the equations of the form $aX^4+bX^2 Y^2+cY^4=dZ^2$ over Gaussian integers by a method of the resolvents. We study as a new equations $X^4+6X^2 Y^2+Y^4=Z^2$ (Mordell's equation over $\mathbb{Z}[i]$), $X^4+6(1+i)X^2Y^2+2iY^4=Z^2$ and $X^4\pm Y^4=(1+ i)Z^2$ and give the new proofs of the known theorems on $X^4+Y^4=Z^2$ (Fermat - Hilbert), $X^4\pm Y^4=iZ^2$ (Szab{\'o}  - Najman).
\end{abstract}

\section{Introduction}
Quartic equations of the form $aX^4+bX^2 Y^2+cY^4=dZ^2$ have a rich past and were research topics to many famous names: P. Fermat, L. Euler, A. Legendre and J. Lagrange. However, their results concerned only equations over the rational integers $\mathbb{Z}$.
 
In a closer past, D. Hilbert, A.Aigner, T. Nagell, L.J. Mordell, B.N. Delone, D.K. Faddeev had undertaken the challenge of studying the algebraic equations over some abelian algebraic extensions of $\mathbb{Z}$ and even ventured into generalizing the problem to having algebraic integers as coefficients. 

The present work is structured as follows. First, there is a discussion of various preliminary results needed to build the future resolvent theory for the quartic equations. For quartic equations, the resolvent is a system of equations of second-degree which is algebraically obtained from the original diophantine equation \cite{Sidokhine:2013ab}. The concept of resolvents is introduced and various specific quartic equations of the form $aX^4+bX^2Y^2+cY^4=dZ^2$ are studied over $\mathbb{Z}[i]$.

The present work itself is a genuinely original approach to the problem of the quartic equations with only trivial solutions, and is only related to other publications on this topic in terms of results but not in terms of methodology.

\section{The Infinite Descent Hypothesis for the Rings with Unique Factorization} 

\begin{definition}
A ring $R$ is a unique factorization domain (UFD) if:

\begin{itemize}
\item $R$ is a domain,
\item any irreducible element of $R$ is also a prime element,
\item for any a belonging to $R$ takes place $a=u p_1 p_2...p_m$  where u belongs to unit group $U$ and $p_i$ belong to a set primes $\pi$ of $R$ not all $p_i$  are necessary distinct, and this representation is unique up to permutations and associates.
 \end{itemize}
 
 \end{definition}
 
\subsection{On Decomposing of the Ring with Unique Factorization into Disjoint Cosets}

An invariant characteristic an element of $R$ is the number of prime divisors contained in its representation including multiplicities. This characteristic for any $a \in R$ will be denoted $\nu(a)$. In practice, $\nu$ is interpreted as a map $\nu: R \to \mathbb{Z}_+= \{n\in Z | n \geq 0\}$. The mapping is built as follows:

\begin{definition}
$\nu_p: R \to \mathbb{Z}_+$  , where $p \in \pi$, acts on any element $a$ of $R$ as follows: $\nu_p(a)=\alpha, p^\alpha || a$.   For any $p$, $\nu_p(a)=0 $ if $p\nmid a$ (i.e. $a \nequiv 0 \mod p$).
\end{definition}

\begin{definition}
$\nu : R\to \mathbb{Z}_+$, $\nu$ acts on any element $a \in R$ as follows: $\nu(a)=\sum_{p\in\pi} \nu_p(a)$.
\end{definition}

$\sum_{p\in\pi} \nu_p (a)$ consists of only a finite number of terms due to $R$ being UFD.
Let us notice some important properties of the mapping $\nu$:

\begin{itemize}
\item $\nu(a) = 0$ $\leftrightarrow$ $a$  belongs to unit group
\item $\nu(a) = 1$ $\leftrightarrow$ $a$  belongs to $\pi$
\item $\nu(ab) = \nu(a) + \nu(b)$
\item if $a | b$ and $b \nmid a$ $\rightarrow$ $\nu(a) < \nu(b)$.
\end{itemize}
         
\begin{proposition} 
$R/\ker (\nu)=\{A_n\}$, $A_n=\{ a \in R | \nu(a)=n\}$ and if $n \neq m$ then $A_n\cap A_m=\emptyset$. 
\end{proposition}

\begin{proof}
Proposition 2.4 is a direct consequence of unique factorization.
\end{proof}

Furthermore, define a binary operation $\mu : R/\ker(\nu) \times R/\ker(\nu) \to R/\ker(\nu)$ as follows:  
$\mu(A_n, A_m)=A_{\nu(xy)}= A_{n+m}$  where $x\in A_n, y\in A_m$. The fact that $\mu$ is well-defined is again a consequence of unique factorization. Defines multiplication on the partition $R/\ker(\nu)$ making it into a monoid.

\begin{theorem}
$(R/\ker(\nu),\mu)$ is a monoid.
\end{theorem}

\begin{proof}
Since $\mu(A_n, A_m)$ is simple denoted $A_n \cdot A_m$  let us use that notation instead. For any $n, m$ and $k$ takes place $(A_n\cdot A_m)\cdot A_k  = A_n\cdot (A_m\cdot A_k)$ There exists an identity element, $A_0$ such that for any $n$,  $A_n \cdot A_0  = A_0 \cdot A_n  = A_n$.
\end{proof}

\begin{theorem}
$R/\ker (\nu)$ has a well - ordering generated by $\nu$. 
\end{theorem}

\begin{proof}
Since ordering (or well - ordering for that matter) is a binary relation defined on the elements of $A/\ker(\nu) \times A/\ker(\nu)$ this simple construction. Let $a \in A_n$ and $a' \in A_m$. Define: $A_n  < A_m$ if $\nu(a) < \nu(a')$. Since $\nu$ is well defined this definition does not depended on which element we chose from $A_n$ and $A_m$. This is a natural ordering of $R/\ker(\nu)$.
\end{proof}

The cosets of $R/\ker(\nu)$ are well-ordered and hence can now be subject to infinite descent. The only additional requirement is a definition of true (and by negation false) statements on cosets.

\subsection{The Principle of Infinite Descent for Cosets}

\begin{definition}
Given a statement $T$ defined on any element from $A_n$ then $T(A_n)$ is true if and only if for any element $a' \in A_n$ the statement $T(a')$ is true. $T(A_n)$ is false if there exists $a'' \in A_n$ such that $T(a'')$ is false.
\end{definition}

Since the cosets are well ordered we can now formulate the infinite descent hypothesis:

\begin{theorem}[The Principle of Infinite Descent for Cosets] 
Let $T$ be a class statement and $T(A_1),..., T(A_n)$ are all true. Suppose now that $T(A_{n+1})$ is false.  If it will follow that from $T(A_{n+1})$ being false will imply that one of the $T(A_k)$ with $k<n+1$ is also false we have a contradiction and $T(A_{n+1})$ must have be true. 
\end{theorem}

\section{The Quartic Equations $aX^4+bX^2 Y^2+cY^4=dZ^2$ over Gaussian Integers}

\subsection{The Equation $X^4+Y^4=Z^2$  (Fermat - Hilbert)}

\begin{theorem}
The equation $X^4+Y^4=Z^2$, where $\gcd(X, Y, Z) \in U$ and $XYZ \neq 0$, has no non-trivial solutions over Gaussian integers.
\end{theorem}

\begin{proof}
Let $(\alpha, \beta, \gamma)$ be some solution of the equation then $\nu(\beta)\geq 3$. Thus the equation has no solution if $\beta$ belongs to $A_0, A_1, A_2$.  Let $A_{n+1}$ be the first when $(\alpha, \beta, \gamma)$ is any solution the equation where $\beta$ belongs to $A_{n+1}$. Let $\alpha, \gamma$ belong to $O^I$, and $\beta=(1+i)^{2+m} p_2^{\alpha_2 }...p_n^{\alpha_n}$, where $p_i$  are distinct primes and belong to $O^I$ (see \cite{Sidokhine:2016aa}). Then, according to \cite{Sidokhine:2016aa},
\begin{equation*} 
\alpha^2=i^{t+1} (P'^2-(-1)^t Q'^2), \beta^2=(1+ i)^2 P'Q', \gamma =i^{t+1} (P'^2+(-1)^t Q'^2)
\end{equation*}
$\gcd(P', Q') = 1, P'Q'\equiv  0 \mod (1+ i)$. Thus $P'=P^2, Q'=Q^2$ where $P$ or $Q$ belongs to $O^I$.
\begin{equation*}
\alpha^2=i^{t+1} (P^4-(-1)^t Q^4) , \beta^2 = (1+ i)^2 P^2 Q^2, \gamma =i^{t+1} (P^4+(-1)^t Q^4)
\end{equation*}

Let $t = 0$ then $\alpha^2=i(P^4-Q^4)$. Since $\alpha$ is odd, $\gcd(P, Q) = 1$ and $PQ\equiv 0 \mod (1+i)$ we can suppose that $P, Q$ are odd and even respectively. Then $R(\alpha^2)+I(P^4)-I(Q^4)=0$. Thus $R(\alpha^2)\equiv 0 \mod 4$ but since $\alpha \in O^I$ we have $R(\alpha^2) \equiv  1 \mod 4$. That is a contradiction. 

Let $t = 1$ then $\alpha^2=-(P^4+Q^4)$ and $R(\alpha^2)+R(P^4)+R(Q^4) = 0$. Then $R(\alpha^2 )+R(P^4 )\equiv 0 \mod 4$ as $\alpha \in O^I$ so $R(\alpha^2)+R(P^4)\equiv 2 \mod 4$. That is a contradiction.

Let $t = 2$ according to case $t = 0$, $R(\alpha^2)\equiv 0 \mod 4$ but $R(\alpha^2)\equiv 1 \mod 4$. That is a contradiction. 

Let $t = 3$ then $P^4+Q^4=\alpha^2$. Thus $(\alpha', \beta', \gamma')$ where $\alpha' = P, \beta' = Q, \gamma' = \alpha$ is the solution of given equation $X^4+Y^4=Z^2$ where $\beta'$ is a proper divisor of $\beta$ so $\nu(\beta')< \nu(\beta)$. Thus $\beta'$ belongs to $A_k<A_{n+1}$. That is a contradiction. 

Thus the equation  $X^4+Y^4=Z^2$  has no non-trivial solutions over $\mathbb{Z}[i]$.
\end{proof}

\subsection{The Equations $X^4+iY^2=Z^4$ ; $X^4+Y^4=iZ^2$ (Szab{\'o}  - Najman)}

\begin{theorem}
The equation $X^4+iY^2=Z^4$ where $\gcd(X, Y, Z)\in U$ and $XYZ \neq 0$, has no non-trivial solutions over Gaussian integers.
\end{theorem}

\begin{proof}
Let $(\alpha, \beta, \gamma)$ be a solution of the equation. Let $\alpha, \gamma$ belong to $O^I$, and $\beta=(1+i)^{2+m} p_2^{\alpha_2}...p_n^{\alpha_n}$, where $p_i$ are distinct primes and belong to $O^I$ be a solution of the equation \cite{Sidokhine:2016aa}, we have $\alpha^2=i^{1-t} (P^2-(-1)^t iQ^2), \beta=(1+ i)^2 PQ, \gamma^2=i^{1-t} (P^2+(-1)^t iQ^2)$.

Let $t = 0$ or $2$ then $\alpha^2=i(P^2-iQ^2), \gamma^2=i(P^2+iQ^2)$ or $\alpha^2=-i(P^2-iQ^2 ), \gamma^2=-i(P^2+iQ^2)$  where $\gcd(P, Q) = 1, PQ \equiv  0 \mod (1+ i)$ and $P$ or $Q$ belongs to $O^I$. Thus 
$\alpha^2=iP^2+Q^2, \gamma^2=iP^2-Q^2$.

The equation $X^4+iY^2 + Z^4=0$, according to \cite{Sidokhine:2016aa}, for the first equality $P\equiv 0 \mod (1+ i)^2$, for the second equality $P\equiv 0 \mod (1+ i)$ and $P\nequiv 0 \mod (1+ i)^2$. That is a contradiction.

Let $t = 1, 3$ then we have $\alpha^2= P^2+iQ^2, \gamma^2= P^2-iQ^2$,
or $\alpha^2=-P^2-iQ^2, \gamma^2=-P^2+iQ^2$. Thus we have $(\alpha\gamma)^2=P^4+Q^4$. This equality is impossible according to theorem 3.1. Thus the equation $X^4+iY^2=Z^4$ has no non-trivial solutions over $\mathbb{Z}[i]$.
\end{proof}

\begin{theorem}
The equation $X^4+Y^4=iZ^2$, where $\gcd(X, Y, Z)\in U$ and $XYZ \neq 0$, has only the finite number of non-trivial solutions.
\end{theorem}

\begin{proof}
Let $(\alpha, \beta, \gamma)$ be a solution of the equation then $\nu(\beta) \geq 3$. The equation has no solution if $\gamma$ belongs to $A_0, A_1, A_2$.  Let $A_{n+1}$ be the first when $(\alpha, \beta, \gamma)$ is a solution the equation, where $\gamma \in A_{n+1}$. According to \cite{Sidokhine:2016aa}, $\alpha, \beta \in O^I$, $\gamma=i^m (1+ i)p_2^{\alpha_2}...p_n^{\alpha_n}$, where $m$ is equal $0$ or $1$ and $p_i$   are distinct primes and belong to $O^I$. Then, according to \cite{Sidokhine:2016aa},
\begin{equation*} 
\alpha^2=i^t \frac{P^2\mp (-1)^t iQ^2}{1+ i}, \beta^2=i^{t+1} \frac{P^2\pm(-1)^t iQ^2}{1+ i}, \gamma = i^{\frac{1 \mp 1}{2}} (1+i)PQ
\end{equation*}
where $P, Q$ belong to $O^I$ and $\gcd(P, Q) = 1$. Let us show that for $t$ is possible value only $0$. Indeed according to the study of the equation $X^2+iY^2+(1+ i) Z^2 = 0$, \cite{Sidokhine:2016aa} , we have
\begin{itemize}
\item Let $t=0$ then $(1+ i) \alpha^2=P^2-iQ^2, (1+ i) \beta^2=iP^2-Q^2$ that it is impossible.
\item Let $t=0$ then $(1+ i) \alpha^2=P^2+iQ^2, (1+ i) \beta^2=iP^2+Q^2$ that it is possible. 
\item Let $t=1$ then $(1+i) \alpha^2=iP^2 \mp Q^2, (1+i) \beta^2=-P^2\pm iQ^2$ that it is impossible. 
\item Let $t= 2$ then $(1+i) \alpha^2=-P^2 \pm iQ^2, (1+i) \beta^2=-iP^2\mp Q^2$ that it is impossible.
\item Let $t= 3$ then $(1+i) \alpha^2=-iP^2\mp Q^2, (1+i) \beta^2=-P^2\mp iQ^2$ that it is impossible. 
\end{itemize}
Thus we have the following system of equalities
\begin{equation*}
\alpha^2=\frac{P^2+iQ^2}{1+ i}, \beta^2=\frac{iP^2+Q^2}{1+ i}, \gamma = i(1+ i)PQ.
\end{equation*}
Then $\alpha\beta, P, Q$ satisfy the following equality
\begin{equation*}
2(\alpha\beta)^2=P^4+Q^4
\end{equation*}
where $\alpha, \beta, P, Q$ belong to $O^I$ and $\gcd(P, Q) = 1$. So we can write the following equality
\begin{equation*}
(\frac{P^2+Q^2}{2})^2+(\frac{P^2-Q^2}{2})^2  =(\alpha\beta)^2
\end{equation*}
where $\gcd( \frac{P^2+Q^2}{2}, \frac{P^2-Q^2}{2}) = 1$ and  $\frac{(P^2+Q^2)(P^2-Q^2)}{4}\equiv  0 \mod(1+ i)^2$. 
Possible two cases: 
\begin{itemize}
\item Case 1: Let $P^2-Q^2= 0$ then $(i^s, i^t, \pm i(1+i))$ where $0 \leq s, t \leq 3$ is the full set of solutions of the equation $X^4+Y^4=iZ^2$.
\item Case 2: $P^2-Q^2 \neq 0$ then equation $X^4+Y^4=iZ^2$ has no solution. Indeed, let us consider the system of the equations
 \end{itemize}
 
\begin{equation*}
\begin{cases}
 (\frac{P^2+Q^2}{2})^2+(\frac{P^2-Q^2}{2})^2 =(\alpha\beta)^2 \\             
  (\frac{P^2+Q^2}{2})^2-(\frac{P^2-Q^2}{2})^2 =(PQ)^2.
\end{cases}
\end{equation*}  
  
Notice. Since $\alpha, \beta, P, Q \in O^I$ and $\gcd(\alpha, \beta) = \gcd(P, Q) = 1$. One can write the following equalities $\frac{P^2+Q^2}{2}=u$, where $u \in O^I$, and $\frac{P^2-Q^2}{2}= i^l v$ where  $l = 0;1, v=(1+ i)^{2+r} w$  and  $w \in O^I$.

Thus we can write the equalities where 
\begin{equation*}
\begin{cases}
u^2\pm v^2=(\alpha\beta)^2 \\
u^2\mp v^2=(PQ)^2
\end{cases}
\end{equation*}

We take only up or down sings of the system equations. According to the study of the equation $X^2+Y^2+Z^2=0$ \cite{Sidokhine:2016aa} we have the following equalities
\begin{eqnarray*}
      u = i^{t+1} (m^2-(-1)^t n^2 ),   v = (1+ i)^2 mn \\
      u = i^{s+1} (m'^2+(-1)^s n'^2),   v = (1+ i)^2 m'n'
\end{eqnarray*}

where $\gcd(m, n) = \gcd(m', n') = 1$, $mn\equiv m'n'\equiv 0 \mod(1+ i)$ and $m$ or $n$ belong to $O^I$ and also and $m'$ or $n'$ belong to $O^I$. Show that  $t-s\equiv 0 \mod 2$. Let  $t-s \nequiv 0 \mod 2$  then $t - s = 2k +1$ and
\begin{equation*} 
i^{s+1+2k+1} (m^2-(-1)^{s+1} n^2 )=i^{s+1} (m'^2+(-1)^s n'^2 )
\end{equation*}
or 			          
\begin{equation*}
 (-1)^k i(m^2+(-1)^s n^2)=(m'^2+(-1)^s n'^2)
 \end{equation*}
But                
\begin{equation*}
(-1)^k R(i(m^2+(-1)^s n^2))=(-1)^k I(m^2)+(-1)^{k+s} I(n^2)\equiv 0 \mod 2		
\end{equation*}
then as
\begin{equation*}
R(m'^2+(-1)^s n'^2)=R(m'^2)+(-1)^s R(n'^2)\equiv 1 \mod 2. 
\end{equation*}
We have a contradiction. Thus $t - s\equiv 0 \mod 2$ and we have
\begin{equation*}
(-1)^k (m^2-(-1)^s n^2 )= m'^2+(-1)^s n'^2.
\end{equation*}
Since $(k, s)$ takes the values $(0, 0), (0, 1), (1, 0), (1, 1)$ they leads to the system of the equations
\begin{equation*}
\begin{cases}
n^2+m^2=n'^2-m'^2 \\
nm = n'm' 
\end{cases}
\end{equation*}

Let $m\equiv m'\equiv 0 \mod(1+ i)$ then $n, n'\in O^I$ since $nm = n'm'$ and $\gcd (n, m) = \gcd(n', m') = 1$.

Let $(n, m, n', m')$ be a solution of the system of equations then $\nu(nm)\geq3$. The system of equations has no solution if the product $nm$ belongs to $A_0, A_1, A_2$. Let $A_{n+1}$ be the first when $(n, m, n', m')$ is a solution the equation where $nm \in A_{n+1}$.

Further we use the notation $\gcd(m, n) = (m, n)$. The following factorization is possible. 
\begin{eqnarray*}
n  =  (n, n')(n, m'),	 m  =  (m, n')(m, m') \\
n' =  (n', n)(n', m),	 m' =  (m', n)(m', m)
\end{eqnarray*}
Let us make substitution 
\begin{equation*}
(n, n')^2 (n, m')^2+(m, n')^2 (m, m')^2=(n', n)^2 (n', m)^2-(m', n)^2 (m', m)^2
\end{equation*}
then we have
\begin{equation*}
(n, m')^2 ((n, n')^2+(m', m)^2)=(n', m)^2 ((n, n')^2-(m, m')^2).
\end{equation*}
Since  $\gcd((n, m'), (m, n')) = \gcd( (n, n')^2+(m', m)^2, (n, n')^2-(m', m)^2) = 1$. Thus
\begin{equation*} 
k = ((n, n')^2-(m', m)^2)/(n, m')^2=((n, n')^2+(m', m)^2)/(m, n')^2
\end{equation*}
where $k = \pm1$ or $\pm i$. However, the value $k$ can not be equal to $(-1)$ or $(\pm i)$. Since $(n, n'), (n, m')$ and $(n', m)$ belong to $O^I$ we have for case 1:
\begin{equation*}
(n', n)^2-(m, m')^2= -(n, m')^2
\end{equation*}
but $(n', n)^2, - (n, m')^2$ must have the same sign. That is a contradiction. For case 2 we have
\begin{equation*}
(n', n)^2-(m, m')^2= \pm i(n, m')^2
\end{equation*}
But $(n,m')$ must be even. That it is contradiction.  Thus we have gotten the system of equalities
\begin{equation*}
\begin{cases}
\alpha'^2=P'^2+Q'^2 \\
\gamma'^2=P'^2-Q'^2
\end{cases}
\end{equation*}
where $\alpha' = (n, m')$, $\gamma'= (m, n')$, $P'= (n', n)$, $Q' = (m', m)$ and $Q'$ is a divisor of the product $nm$.
The expressions for $P', Q'$ of the first equality have the following form 
\begin{equation*}
P' = u^2-v^2, Q'=(1+i)^2 uv, \gcd(u, v) = 1.
\end{equation*}
The expressions for $P', Q'$ of the second equality have the following form
\begin{equation*}
P' = u'^2+v'^2, Q'=(1+i)^2 u'v', \gcd(u', v') = 1.
\end{equation*}
Thus we have the following system of the equalities
\begin{equation*}
\begin{cases}
u^2+v^2  = u'^2-v'^2 \\  
uv = u'v' \\
\gcd(u, v) = \gcd(u', v') = 1
\end{cases}
\end{equation*}
where $(u, v, u', v')$ is a solution of the system of equations. Since the product $uv$ is a proper divisor $Q'$ so $\nu(uv) < \nu(nm)$. Thus the product $uv$ belongs to $A_k<A_{n+1}$. That is a contradiction. Thus the equation $X^4+Y^4=iZ^2$ has only a finite number of the non-trivial solutions over $\mathbb{Z}[i]$.
\end{proof}

\subsection{The Equations $X^4+(1+i) Y^2=Z^4$; $X^4+Y^4=(1+i) Z^2$}

\begin{theorem}
The equation $X^4+(1+i) Y^2=Z^4$, where $\gcd(X, Y, Z)\in U$ and $XYZ \neq 0$, has no non-trivial solutions over Gaussian integers.
\end{theorem}

\begin{proof}
Let $(\alpha, \beta, \gamma)$ be a solution of the equation then $\nu(\beta) \geq 3$. Thus the equation has no solution if to $\beta \in A_0, A_1$ or $A_2$.  Let $A_{n+1}$ be the first when $(\alpha, \beta, \gamma)$ is a solution the equation, where $\beta \in A_{n+1}$. Let $\alpha, \gamma$ belong to $O^I$, and $\gamma=i^{2+m} (1+ i)p_2^{\alpha_2}...p_n^{\alpha_n}$, where $p_i$ are distinct primes and belong to $O^I$. Then, according to \cite{Sidokhine:2016aa}, we have
\begin{equation*}
\alpha^2=i^{1-t} (P^2-(-1)^t (1+i) Q^2), \beta = (1+i)^2 PQ, \gamma^2=i^{1-t} (P^2+(-1)^t (1+i) Q^2).
\end{equation*}
Let $t = 0$ then
\begin{equation*} 
\alpha^2=i(P^2-(1+i) Q^2), \gamma^2=i(P^2+(1+i) Q^2) \text{ and } (\alpha\gamma)^2  = -P^4+(1+ i)^2 Q^4
\end{equation*}

Since $\alpha, \gamma, P$ belong to $O^I$ so $(\alpha\gamma)^2, -P^4$ must have the same sign. That is a contradiction. 

Let $t = 2$ like $t = 0$, $(\alpha\gamma)^2, -P^4$ must have the same sign. That is a contradiction.

Let $t = 3$ then $\alpha^2= -(P^2+(1+ i) Q^2), \gamma^2= -(P^2-(1+ i) Q^2)$.

Since $\alpha, P$ belong to $O^I$ so $\alpha^2, -P^2$ must have the same sign. That is a contradiction. 

Let $t =1$ then we have the following system of equations
\begin{equation*}
\begin{cases}
P^2+(1+i) Q^2=\alpha^2 \\ 
P^2-(1+i) Q^2=\gamma^2
\end{cases}
\end{equation*}
According to \cite{Sidokhine:2016aa}, the individual solutions of each of the equations have the following forms
\begin{eqnarray*}
P=i^{1-t} (n^2-(-1)^t (1+i) m^2),   Q = (1+ i)^2 nm,   \alpha = i^{1-t} (n^2+(-1)^t (1+i) m^2) \\
P=i^{1-s} (n'^2+(-1)^s (1+i) m'^2),   Q = (1+ i)^2 n'm',   \gamma = i^{1-s} (n'^2-(-1)^s (1+i) m'^2)
\end{eqnarray*}

where $\gcd(n, m) = \gcd(n', m') = 1$ and $n, n' \in O^I$. Show that $t - s\equiv 0 \mod 2$. Let $t - s \nequiv 0 \mod 2$ then $s - t = 2k +1$ and
\begin{equation*} 
i^{1-s+2k+1} (n^2-(-1)^{s+1}(1+i) m^2)=i^{1-s} (n'^2+(-1)^s (1+i)m'^2)
\end{equation*}
or 			
\begin{equation*}
(-1)^k i(n^2+(-1)^{s}) (1+i) m^2)=(n'^2+(-1)^s (1+i)m'^2)
\end{equation*}
But $(-1)^k R(i(n^2+(-1)^s (1+i) m^2))=(-1)^k I(n^2)+(-1)^{s+k} (I(m^2 )-R(m^2))\equiv 0 \mod 2$ then as $R(n'^2+(-1)^s (1+i)m'^2)=R(n'^2)+(-1)^s (R(m^2)-I(m^2 ))\equiv 1 \mod 2$. That is a contradiction. 
Thus  we have $(-1)^k (n^2-(-1)^s (1+ i) m^2)=n'^2+(-1)^s (1+ i)m'^2$. Given result leads to the following system of   the equations 
\begin{equation*}
\begin{cases}
n^2+(1+ i) m^2=n'^2-(1+ i)m'^2 \\
nm = n'm'
\end{cases}
\end{equation*}          

where $n, n'\in O^I$ since $m\equiv m'\equiv 0 \mod (1+i)$ and $\gcd(n, m) = \gcd(n', m') = 1$.

Let $(n, m, n', m')$ be a solution of the system of equations then $\nu(nm) \geq 3$. The system of equations has no solutions if the product $nm$ belongs to $A_0, A_1, A_2$.  Let $A_{n+1}$ be the first when $(n, m, n', m')$ is any solution the equation where $nm \in A_{n+1}$. Further we use the notation $\gcd(n, m) = (n, m)$. The following factorization is possible
\begin{eqnarray*}
n  = (n, n')(n, m'),	 m =  (m, n')(m, m') \\
n' = (n', n)(n', m),	 m' =  (m', n)(m', m)
\end{eqnarray*}
Let us make substitution
\begin{equation*}
(n, n')^2 (n, m')^2+(1+i) (m, n')^2 (m, m')^2=(n', n)^2 (n', m)^2-(1+ i) (m', n)^2 (m', m)^2
\end{equation*}
then we have
\begin{equation*}
(n, m')^2 ((n, n')^2+(1+i) (m', m)^2)=(m, n')^2 ((n', n)^2-(1+i) (m, m')^2)
\end{equation*} 
Since  $\gcd((n, m'), (m, n')) = \gcd((n, n')^2+(1+i) (m', m)^2, (n', n)^2-(1+i) (m, m')^2) = 1$. Thus
\begin{equation*} 
k = ((n', n)^2-(1+i) (m, m')^2)/(n, m')^2=((n, n')^2+(1+i) (m', m)^2)/(m, n')^2
\end{equation*}
where $k = \pm1$ or $\pm i$. However, the value $k$ cannot be equal to $(-1)$ or $\pm i$. Since $(n, n'), (n, m')$ and $(n', m)$ belong to $O^I$  we have for case 1
\begin{equation*}
(n', n)^2-(1+ i) (m, m')^2= -(n, m')^2
\end{equation*}
but $(n', n)^2, -(n, m')^2$ must have the same sign. That is a contradiction. For case 2 we have
\begin{equation*}
(n', n)^2+(1+ i) (m, m')^2= \pm i(n', m)^2
\end{equation*}
According to \cite{Sidokhine:2016aa}, $(m', m)$ has to belong to $O^I$  but $(m', m)\equiv 0 \mod(1+ i)$. That is a contradiction. 
Thus we have gotten the following system of the equalities
\begin{equation*}
\begin{cases}
P'^2+(1+i) Q'^2=\alpha'^2 \\
P'^2-(1+i) Q'^2=\gamma'^2
\end{cases}
\end{equation*}  
where $P'= (n', n), Q' = (m', m), \alpha' = (n, m'), \gamma'= (m, n')$ and $Q'$ is a divisor of the product nm.
The expressions for $P', Q'$ of the first equality have the following form 
\begin{equation*}
P'=u^2-(1+i) v^2, Q'=(1+ i)^2 uv,\gcd(u, v) = 1 
\end{equation*}
The expressions for $P', Q'$ of the second equality have the following form 
\begin{equation*}
P'=u'^2+(1+i) v'^2, Q'=(1+ i)^2 u'v', \gcd(u', v') = 1
\end{equation*}
Thus we have the following system of the equalities
\begin{equation*}
u^2+(1+i) v^2=u'^2-(1+i) v'^2,uv = u'v' \\
\gcd(u, v) = \gcd(u', v') = 1
\end{equation*}
where $(u, v, u', v')$ is a solution of the system of equations. Since the product $uv$ is a proper divisor $Q'$ so $\nu(uv) < \nu(nm)$. Thus the product $uv$ belongs to $A_k<A_{n+1}$. That is a contradiction. Thus the equation $X^4+(1+i) Y^2=Z^4$ has no non-trivial solutions over $\mathbb{Z}[i]$.
\end{proof}

\begin{theorem}
The equation $X^4+Y^4=(1+i)Z^2$, where $\gcd(X, Y, Z)\in U$ and $XYZ \neq 0$, has no non-trivial solutions over Gaussian integers.
\end{theorem}

\begin{proof}
Since properties of solutions of the equation $X^2+(1+i)Y^2+Z^2=0$ are not compatible with solvability of the quartic equation so $X^4+Y^4=(1+i)Z^2$ has no solution over $\mathbb{Z}[i]$. 
\end{proof}

\subsection{The Equation $X^4+6X^2 Y^2+Y^4=Z^2$ (Mordell)}

\begin{theorem}
The equation $X^4+6X^2 Y^2+Y^4=Z^2$, where $\gcd(X, Y, Z)\in U$ and $XYZ \neq 0$, has only the finite number of non-trivial solutions over Gaussian integers.
\end{theorem}

Let $X, Y$ belong to $G$ - submonoid in other words $X = p_2^{\alpha_2}...p_n^{\alpha_n}, Y=(1+ i)^\beta p_2^{\beta_2}...p_m^{\beta_m}$ , where $p_i$ are distinct primes and belong to $O^I$ then we should consider two equations
\begin{equation*}
X^4\pm6X^2 Y^2+Y^4=Z^2 \text{ where } X, Y\in G
\end{equation*}
We study the case $X^4+6X^2 Y^2+Y^4=Z^2$ other case can be considered the same way. 

\begin{lemma}
The equation $X^4+6X^2 Y^2+Y^4=Z^2$, where $\gcd(X, Y, Z)=1$, $XY\equiv 0 \mod(1+i)$, $XYZ \neq 0$ and $X, Y\in G$, has a solution if only if the system of the equations
\begin{equation*}
\begin{cases}
U^2+V^2=U'^2-V'^2 \\
UV = U'V' \\
\gcd(U, V) = \gcd(U', V') =1
\end{cases}
\end{equation*} 
where $U, V, U', V'$ belong to $G$-set, has a solution.
\end{lemma}

\begin{proof}
Let $u_0, v_0, u_0', v_0'$ be a solution of the system of equations, where $u_0 v_0\equiv 0 \mod(1+i)$ and $u_0, v_0, u_0', v_0'$ belong to $G$ - submonoid, then the equation
\begin{equation*} 
z^2- (u_0^2+v_0^2)z -(u_0 v_0 )^2=0
\end{equation*}
has a solution over Gaussian integers and so the discriminate of this equation must be faithful square. In other words there exists $d_0$ such that 
\begin{equation*}
d_0^2= u_0^4  +6u_0^2 v_0^2+v_0^4.
\end{equation*}
Let $(u_0, v_0, d_0)$ be a solution of the equation $X^4+6X^2 Y^2+Y^4=Z^2$, where $\gcd(u_0, v_0)=1$, $u_0 v_0 d_0 \neq 0, u_0 v_0\equiv 0 \mod(1+ i)$ and $u_0 ,v_0, d_0 \in G$ then $u_0' =\frac{d_0+u_0^2+v_0^2}{2}, v_0' =\frac{d_0-u_0^2-v_0^2}{2}$, where $u_0', v_0'$ are Gaussian integers and $\gcd(u_0', v_0') \in U$. Since $u_0'v_0' = (u_0 v_0 )^2$, where $u_0, v_0 \in G$, we can write $u_0'=i^s n^2$, $v_0'=i^{-s} m^2$, where $n, m\in G$, $\gcd(n, m)=1$, and $nm =u_0 v_0$. According to \cite{Sidokhine:2016aa}, for $s$ there are only two possible values $0,2$ and we can write down the following equalities
\begin{equation*}
\begin{cases}
u_0^2+v_0^2  = n^2- m^2 \\
u_0 v_0  = nm \\
\gcd(u_0, v_0)=\gcd(n, m)=1
\end{cases}
\end{equation*}
Since the system of equations of second - degree has no solution so the quartic equation also has no non-trivial solutions. 
\end{proof}

\begin{lemma}
The equation $X^4+6X^2 Y^2+Y^4=Z^2$, where $\gcd(X, Y, Z)=1,  XYZ \neq0 , XY\nequiv 0 \mod(1+ i)$ and  $X, Y \in G$ has only the finite number of non-trivial solutions over Gaussian integers.
\end{lemma}

\begin{proof}
Let $(u_0, v_0, d_0), \gcd(u_0, v_0)=1, u_0 v_0 d_0\neq0, u_0 v_0\nequiv0 \mod(1+ i)$  and $u_0 ,v_0\in G$ be a solution of equation.
\begin{equation*}
d_0^2= u_0^4  +6u_0^2 v_0^2+v_0^4= (u_0^2+v_0^2 )^2-((1+i)^2 u_0 v_0 )^2 
\end{equation*}
Let $u_0^2+v_0^2  =0$ then given equation has solutions $(\pm1, \pm i, \pm i(1+ i)^2), (\pm i, \pm1, \pm i(1+ i)^2)$, where the choice of the signs is arbitrary, are all solutions of given equation. Let $u_0^2+v_0^2  \neq 0$ then, according to \cite{Sidokhine:2016aa}, this equation has no solution.
\end{proof}

\begin{proof}[Proof of theorem 3.6]
According lemmas 3.7, 3.8, the equation $X^4+6X^2 Y^2+Y^4=Z^2$, where $XYZ \neq 0$ and $\gcd(X, Y, Z)\in U$ and $X,Y \in G$ has only the finite number of non-trivial solutions over $\mathbb{Z}[i]$. 
\end{proof}

\subsection{The Equation $X^4+6(1+i) X^2 Y^2+2iY^4=Z^2$} 

\begin{theorem}
The equation $X^4+6(1+i) X^2 Y^2+2iY^4=Z^2$, where $XYZ \neq 0$, $Y \equiv 0 \mod (1+i)$ and $\gcd(X, Y, Z)\in U$, has no non-trivial solutions over Gaussian integers. 
\end{theorem}
Let $X, Y$ belong to $G$-submonoid in other words $X=p_2^{\alpha_2}...p_n^{\alpha_n}, Y=(1+ i)^\beta p_2^{\beta_2}...p_m^{\beta_m}$, where $p_i$  are distinct primes and belong to $O^I$ so we should consider two equations
\begin{equation*}
X^4\pm6(1+i) X^2 Y^2+2iY^4=Z^2, where  X, Y \in G.
\end{equation*}
We study the case $X^4+6(1+i) X^2 Y^2+2iY^4=Z^2$  other case can be considered the same way.
\begin{lemma}
The equation $X^4+6(1+i) X^2 Y^2+2iY^4=Z^2$  where $XYZ \neq 0, Y\equiv 0 \mod(1+i), \gcd(X, Y, Z)=1$ and $X, Y\in G$, has a solution if only if the system of equations 
\begin{equation*}
\begin{cases}
U^2+(1+i)V^2=U'^2-(1+i)V'^2 \\
UV = U'V' \\
\gcd(U, V) = \gcd(U', V') =1
\end{cases}
\end{equation*}
where $U, V, U', V'$ belong to $G$-submonoid and $V \equiv 0 \mod (1+i)$ has a solution.
\end{lemma}

\begin{proof}
Let $u_0, v_0, u_0', v_0'$ be a solution of the system of equations and $u_0, v_0, u_0', v_0'$ belong to $G$-submonoid then the equation
\begin{equation*}
z^2- (u_0^2+(1+i) v_0^2 )z -(1+ i)(u_0 v_0 )^2=0
\end{equation*}
has the solutions over Gaussian integers and so the discriminate of this equation must be faithful square. Other words there exists Gaussian integer $d_0$ such that 
\begin{equation*}
d_0^2= u_0^4  +6(1+i)u_0^2 v_0^2+2iv_0^4
\end{equation*}
$(u_0, v_0, d_0), \gcd(u_0, v_0)=1, u_0 v_0 d_0 \neq 0,  v_0 \equiv 0 \mod(1+ i)$ and $u_0 ,v_0, d_0 \in G$. 
Let $(u_0, v_0, d_0)$ be a solution of $X^4+6(1+i) X^2 Y^2+2iY^4=Z^2$, where $u_0 ,v_0\in G$, $\gcd(u_0, v_0)=1,u_0 v_0 d_0 \neq 0$, then $u_0'=\frac{d_0+u_0^2+(1+ i)v_0^2}{2}, v_0'= \frac{d_0-u_0^2-(1+ i)v_0^2}{2}$,where $u_0', v_0'$ are Gaussian integers and $\gcd(u_0', v_0')\in U$. Since $u_0' v_0'= (1+i)(u_0 v_0 )^2$, where $u_0, v_0 \in G$, we can write $u_0'=i^s n^2, v_0'=i^{-s} (1+i) m^2$, where $n, m\in G$, $\gcd(n, m)=1$, and $nm =u_0 v_0$. According to \cite{Sidokhine:2016aa}, $s$ can only be $0$ or $2$ and we can write down the following equalities
\begin{equation*}
\begin{cases}
u_0^2+(1+ i)v_0^2  = n^2-(1+ i)m^2 \\
u_0 v_0  = nm \\
\gcd(u_0, v_0)=\gcd(n, m)=1
\end{cases}
\end{equation*}

Since the system of equations has no solution so the quartic equation also has no non-trivial solutions. 
\end{proof}

\begin{proof}[Proof of theorem 3.9]
According lemma 3.10 the equation $X^4+6(1+i) X^2 Y^2+2iY^4=Z^2$, where $XYZ \neq 0, Y \equiv 0 \mod (1+i), \gcd(X, Y, Z) \in U$ and $X,Y \in G$, has no non-trivial solutions over $\mathbb{Z}[i]$.
\end{proof}

\section{Discussion and Conclusion}

In our work we have represented the study on quartic equations $aX^4 + bX^2Y^2 + cY^4 = dZ^2$ over the ring of Gaussian integers by a method of the resolvents. By this method one has been gotten the proofs of the propositions both on new quartic equations: $X^4 + 6X^2Y^2 + Y^4 = Z^2$ (Mordell's equation over Gaussian integers), $X^4 + 6(1+i)X^2Y^2 + 2iY^2 = Z^2$, $X^4 \pm Y^4 = (1+i) Z^2$  and the new proofs of the known theorems on the Fermat's quartic equations: $X^4 + Y^4 = Z^2$ (Fermat - Hilbert), $X^4 \pm Y^4 = iZ^2$ (Szab{\'o} - Najman)\cite{Szabo:2004aa}.

For proving the propositions on the quartic equations over the ring $\mathbb{Z}[i]$ one had been built a special form of the fundamental theorem arithmetic, modified Mordell's lemma with account of the properties of Gaussian integers and developed the infinite descent hypothesis for the rings with unique factorization. It has been shown that the results obtained are consistent with known results on Fermat's quartic equations which were proved by evaluating of the norms of Gaussian integers \cite{Cross:1993aa} as well as by the elliptic curves \cite{Najman:2010aa}. The given approach permits from the start to understand why some phenomenon in these problems as example the existence of the finite number of non-trivial solutions for some quartic equations takes place. General problem of a solvability of Fermat's quartic equations over imaginary quadratic rings (Aigner's problem) is one of the topical tasks, \cite{Mordell:1969aa} and \cite{Lynch:2014aa}.

\bibliographystyle{IEEEtran}
\bibliography{references}

\begin{thebibliography}{1}
\providecommand{\url}[1]{#1}
\csname url@samestyle\endcsname
\providecommand{\newblock}{\relax}
\providecommand{\bibinfo}[2]{#2}
\providecommand{\BIBentrySTDinterwordspacing}{\spaceskip=0pt\relax}
\providecommand{\BIBentryALTinterwordstretchfactor}{4}
\providecommand{\BIBentryALTinterwordspacing}{\spaceskip=\fontdimen2\font plus
\BIBentryALTinterwordstretchfactor\fontdimen3\font minus
  \fontdimen4\font\relax}
\providecommand{\BIBforeignlanguage}[2]{{%
\expandafter\ifx\csname l@#1\endcsname\relax
\typeout{** WARNING: IEEEtran.bst: No hyphenation pattern has been}%
\typeout{** loaded for the language `#1'. Using the pattern for}%
\typeout{** the default language instead.}%
\else
\language=\csname l@#1\endcsname
\fi
#2}}
\providecommand{\BIBdecl}{\relax}
\BIBdecl

\bibitem{Sidokhine:2013ab}
F.~Sidokhine, ``A note on quartic equations with only trivial solutions,''
  arXiv preprint, arXiv: 1311.1451, 2013.

\bibitem{Sidokhine:2016aa}
------, ``Quadratic equations in three variables over gaussian integers,''
  arXiv preprint, arXiv: 1607.07386, 2016.

\bibitem{Szabo:2004aa}
S.~Szab{\'o}, ``Some fourth degree diophantine equations in gaussian
  integers,'' \emph{Integers}, 2004.

\bibitem{Cross:1993aa}
J.~T. Cross, ``In the gaussian integers $x^4+ y^4 \neq z^4$,''
  \emph{Mathematics Magazine}, no.~66, pp. 105--108, 1993.

\bibitem{Najman:2010aa}
F.~Najman, ``The diophantine equation $x^4 \pm y^4= iz^2$,'' \emph{Amer. Math.
  Monthly}, no. 117, pp. 637--641, 2010.

\bibitem{Mordell:1969aa}
L.~Mordell, \emph{Diophantine Equations}.\hskip 1em plus 0.5em minus
  0.4em\relax Academic Press Inc., 1969.

\bibitem{Lynch:2014aa}
R.~Lynch and P.~Morton, ``The quartic fermat equation in hilbert class fields
  of imaginary quadratic fields,'' arXiv: 1410.3008v1, October 2014.

\end{thebibliography}

\end{document}